\documentclass{svjour3}                     % onecolumn (standard format)
\usepackage{graphicx}
 \usepackage{mathptmx}      % use Times fonts if available on your TeX system
\usepackage{latexsym}
\usepackage{amsmath,amstext,amssymb}
\usepackage{graphicx}

\date{}

\begin{document}
\title{Extragradient algorithms for split equilibrium problem and nonexpansive mapping}

%\thanks{The first author is supported in part by NAFOSTED, under the project 101.01-2014-24.}

\author{Bui Van Dinh \and Dang Xuan Son \and Tran Viet Anh}
 %      


\institute{Bui Van Dinh \at
              Faculty of Information Technology, Le Quy Don Technical University, Hanoi, Vietnam.\\
                      \email{vandinhb@gmail.com}\\        
             \emph{Present address:} Department of Applied Mathematics, Pukyong National University, Busan, Korea.  
           \and
           Dang Xuan Son \at
              Technical University, Hanoi, Vietnam.
              \email{dangsonmath@gmail.com}
              \and Tran Viet Anh \at
              Department of Scientific Fundamentals, Posts and Telecommunications Institute of Technology, Hanoi, Vietnam.
              \email{tranvietanh@outlook.com} }
\date{Received: date / Accepted: date}
\maketitle

\begin{abstract}
In this paper, we propose new extragradient algorithms for solving a split equilibrium and nonexpansive mapping SEPNM($C, Q, A, f, g, S, T)$ where $C, Q$ are nonempty closed convex subsets in real Hilbert spaces $\mathcal{H}_1, \mathcal{H}_2 $ respectively, $A : \mathcal{H}_1 \to \mathcal{H}_2$ is a bounded linear operator, $f$ is a pseudomonotone bifunction on $C$ and $g$ is a monotone bifunction on $Q$, $S, T$ are nonexpansive mappings on $C$ and $Q$ respectively. By using extragradient method combining with cutting techniques, we obtain algorithms for the problem. Under certain conditions on parameters, the iteration sequences generated by the algorithms are proved to be weakly and strongly convergent to a solution of this problem.

%First keyword \and Second keyword \and More}
 \subclass{$47$H$09$ \and $47$J$25$ \and $65$K$10$ \and $65$K$15$ \and $90$C$99$}
\end{abstract}

\section{Introduction}

Throughout the paper, unless otherwise stated, let we assume that $\mathcal{H}_1$ and $\mathcal{H}_2$ are real Hilbert spaces with inner product denoted by $\langle \cdot , \cdot \rangle$ and associated norm $\| \cdot \|$. By `$\to$' we denote the strong convergence. While `$\rightharpoonup$' stands for the weak convergence.
 Let $C$ be a nonempty closed convex subset of $\mathcal{H}_1$, $Q$ a nonempty closed convex subset of $\mathcal{H}_2$, $A: \mathcal{H}_1 \to  \mathcal{H}_2$ a bounded linear operator. Let $f : C \times C \to \mathbb{R}$ be an equilibrium bifunction, i.e., $f(x, x) = 0$ for every $x \in C$, $g : Q \times Q \to \mathbb{R}$ be an equilibrium bifunction, $S: C \to C$,  and $T: Q \to Q$, are nonexpansive mappings. We consider the following split equilibrium problem and nonexpansive mapping SEPNM($C, Q, A, f, g, S, T)$ ((SEPNM) for short):
\begin{equation} \label{1.1}
 \text{Find } x^* \in C \text{ such that} \begin{cases}
x^* \in \text{Sol}(C, f) \cap \text{Fix}(S)\\
Ax^* \in \text{Sol}(Q, g) \cap \text{Fix}(T),
\end{cases}
\end{equation}
where $\text{Sol}(C, f)$ is the solution set of the following equilibrium problem (shortly EP$(C, f)$)
\begin{equation}\label{1.2}
\text{Find } \bar{x} \in C \text{ such that }  f(\bar{x}, y) \geq 0, \  \forall y \in C.
\end{equation}
$\text{Sol}(Q, g)$ is the solution set of the equilibrium problem EP$(Q, g)$, i.e.,
$$\text{Sol}(Q, g) =\{ \bar{u} \in Q : \ g(\bar{u}, v) \geq 0, \  \forall v \in Q\}.$$
Fix$(S)$ is the fixed point set of mapping $S$, i.e., $\text{Fix}(S) =  \{\bar{x} \in C: S(\bar{x}) = \bar{x} \}$,
and Fix$(T)$ is the fixed point set of mapping $T$, i.e., $ \text{Fix}(T) =  \{\bar{u} \in Q : T(\bar{u}) = \bar{u} \}.$

If $S \equiv I_{\mathcal{H}_1}$, the identity mapping in $\mathcal{H}_1$, and $T \equiv I_{\mathcal{H}_2}$, the identity mapping in $\mathcal{H}_2$, then problem (SEPNM) becomes the following split equilibrium problem (shortly (SEP)):
\begin{equation} \label{1.2}
 \text{Find } x^* \in C \text{ such that } \begin{cases}
x^* \in \text{Sol}(C, f) \\
Ax^* \in \text{Sol}(Q, g),
\end{cases}
\end{equation}
which was introduced by A. Moudafi \cite{Mou1} (see also \cite{DKK,He,KR2}). When $f(x, y) = \langle F(x), y-x\rangle$; $g(u, v) = \langle G(u), v-u \rangle$, with some mappings $F: C \to \mathcal{H}_1$, $G: Q \to \mathcal{H}_2$, (SEP) reduces to the split variational inequality problem (SVIP):
\begin{equation} \label{1.3}
 \text{Find } x^* \in C \text{ such that } \begin{cases}
x^* \text{ solves } \  \text{VIP}(C, F) \\
Ax^* \text{ solves} \ \text{VIP}(Q, G),
\end{cases}
\end{equation}
that was investigated by Y. Censor et al. \cite{CGR}. If $f \equiv 0$ and $g \equiv 0$, then (SEPNM) becomes split common fixed point problem (SFPP), that is to find
\begin{equation}\label{1.4}
x^* \in \text{Fix}(S) \text{ such that }  Ax^* \in \text{Fix}(T),
\end{equation}
which was introduced by A. Moudafi \cite{Mou2} (see also \cite{CS,KS}).

 These problems could be considered as generalizations of split feasibility problem (SFP) which was considered by Y. Censor et al. \cite{CE}. Many iterative methods have been investigated for solving (SFP) and related problems \cite{BB,Byr,CW,DG,LWX}.

For obtaining a solution of (SEP), Z. He \cite{He} introduced an iterative method, which generates a sequence $\{x^k\}$ by
\begin{equation}
\begin{cases}
x^0 \in C; \{r_k\} \subset (0, +\infty); \mu > 0,\\
f(y^k, y) + \displaystyle\frac{1}{r_k} \langle y - y^k, y^k - x^k \rangle \geq 0, \ \forall y \in C, \\
g(u^k, v) + \displaystyle\frac{1}{r_k} \langle v - u^k, u^k - Ay^k \rangle \geq 0, \  \forall v \in Q,\\
x^{k+1} = P_C(y^k + \mu A^*(u^k - Ay^k)), \ \forall k \geq 0.
\end{cases}
\end{equation}

Under certain conditions on parameters, the author shows that $\{x^k\}, \  \{y^k\}$ weakly converges to a solution of (SEP) provided that $f, g$ are monotone bifunctions on $C$ and $Q$ respectively.

To find a common element of the set of fixed points of a nonexpansive mapping and the set of solutions of a pseudomonotone and Lipschitz-type continuous equilibrium problem, P. N. Anh \cite{Anh} proposed to use the sequence which is a combination of extragradient algorithm for equilibrium problem \cite{MQH}, (see also \cite{DM,FP,Kor1}, for more detail extragradient algorithms) and the iterative scheme by the viscosity approximation method \cite{TT},  defined by

\begin{equation}\label{1.5}
\begin{cases}
x^0 \in C; \{\lambda_k\} \subset (0, +\infty),\\
y^k = \arg\min\{\lambda_k f(x^k, y) + \displaystyle\frac{1}{2} \|y - x^k\|: \  y \in C\}, \\
z^k = \arg\min\{\lambda_k f(y^k, y) + \displaystyle\frac{1}{  2} \|y - x^k\|: \  y \in C\},\\
x^{k+1} = \lambda_kx^0 + (1-\lambda_k)S(z^k), \ \forall k \geq 0,
\end{cases}
\end{equation}

It was proved in \cite{Anh} that if $\lim\limits_{k \to \infty}\|x^{k+1} - x^k\| = 0$, then the $ \{x^k\}, \ \{y^k\}, \ \{z^k\}$ generated by (\ref{1.5}) strongly converge to the same point $P_{\text{Sol}(C, f) \cap \text{Fix}(S)}(x^0)$. One advantage of this algorithm is that it could be applied for pseudomonotone equilibrium problems and at each iteration we only solve two strongly convex optimization problems instead of a regularized equilibrium problem.

%From computation point of view this algorithm seem to be easier to apply than
%KR1
 Motivated by above papers and recent works \cite{DKK,KR1,KR2,VSN}, in this paper, we present algorithms for (SEPNM) when $f$ is pseudomonotone and $g$ is monotone, $S, \ T$ are nonexpansive mappings. The rest of paper organizes as follows. In the next Section, some preliminary results are recalled. The weak convergence theorem and its corollary are presented in Section 3. In the last Section, we combine the method presented in Section 3 with the hybrid projection method for obtaining the strong convergence theorem.  A special case of (SEPNM)  is also considered.
 
\section{Preliminaries}
Let $\mathcal{H}$ be a real Hilbert space with an inner product $\langle\cdot,\cdot\rangle$ and induced norm $\Vert \cdot \Vert$, and let $C$ be
a nonempty closed convex subset of $\mathcal{H}$.
By $P_C$, we denote the metric projection operator onto $C$, that is
\begin{equation*}
 P_C(x) \in C: \Vert x - P_C(x)\Vert \leq  \Vert x - y \Vert, \  \forall y \in C.
\end{equation*}

The following well known results will be used in the sequel.

\begin{lemma}\label{2.1} Suppose that $C$ is a nonempty closed convex subset in $\mathcal{H}$. Then $P_C$ has following properties:
\begin{itemize}
\item[(a)] $P_C(x)$ is singleton and well defined for every x;
\item[(b)] $z=P_C(x)$ if and only if $\langle x-z, y-z\rangle\leq 0$, \ $\forall y\in C$;
\item[(c)] $\Vert P_C(x)-P_C(y) \Vert^2 \leq \langle P_C(x) - P_C(y), x-y \rangle, \ \forall x, y \in \mathcal{H}$;
%$\Vert y-P_C(x)\Vert^2 \leq \Vert x-y \Vert^2 - \Vert x-P_C(x)\Vert^2 \ \forall y\in C$;
\item[(d)] $\Vert P_C(x)-P_C(y)\Vert^2 \leq \Vert x-y \Vert^2 - \Vert x-P_C(x) - y + P_C(y) \Vert^2, \ \forall x, y \in \mathcal{H}$.
\end{itemize}
\end{lemma}

\begin{lemma}\label{2.5}
Let $\mathcal{H}$ be a real Hilbert space, then for all $x, y\in \mathcal{H}$ and $\alpha\in [0, 1]$, we have
$$\Vert \alpha x + (1-\alpha)y \Vert^2 = \alpha \Vert x \Vert^2+ (1-\alpha) \Vert y \Vert^2-\alpha(1-\alpha)\Vert x-y\Vert^2.$$
\end{lemma}

\begin{lemma}[Opial's condition]\label{2.4}
For any sequence $\{x^k\}\subset \mathcal{H}$ with $x^k\rightharpoonup x$, the inequality
$$\liminf\limits_{k\longrightarrow+\infty}\Vert x^k-x \Vert < \liminf\limits_{k \longrightarrow +\infty} \Vert x^k-y \Vert$$
holds for each $y\in \mathcal{H}$ with $y \neq x$.
\end{lemma}

Now, we assume that the bifunctions  $g : Q \times Q \to \Bbb{R}$ and $f: C \times C \to \Bbb{R}$ satisfy the following assumptions, respectively:\\
{\bf Assumptions $A$}
\begin{itemize}
\item[$(A_1)$] $g(u, u) = 0$, for all $u \in Q$;
\item[$(A_2)$] $g$ is monotone on $Q$, i.e., $g(u, v) + g(v, u) \leq 0$, for all $u, v \in Q$;
\item[$(A_3)$] for each $u, v, w \in Q$
$$\lim\sup\limits_{\lambda\downarrow 0}g(\lambda w+(1-\lambda)u, v)\leq g(u, v);$$
\item[$(A_4)$] $g(u, . )$ is convex and lower semicontinuous on $Q$ for each $u \in Q$.
\end{itemize}
{\bf Assumptions $B$}
\begin{itemize}
\item[$(B_1)$] $f(x, x)=0$ for all $x\in C$;
\item[$(B_2)$] $f$ is pseudomonotone on $C$, i.e., if $f(x, y)\geq 0$ implies $f(y, x)\leq 0$, for all $x, y\in C$;
\item[$(B_3)$] $f$ is jointly weakly continuous on $C\times C$ in the sense that, if $x, y\in C$ and $\{x^k\}, \{y^k\}\subset C$ converge weakly to $x$ and $y$, respectively, then $f(x^k, y^k) \to f(x, y)$ as $k \to +\infty$;
\item[$(B_4)$]  $f(x, \cdot)$ is convex, subdifferentiable on $C$, for all $x\in C$;
\item [$(B_5)$] $f$ is Lipschitz-type continuous on $C$ with constants $c_1>0$ and $c_2>0$, i.e.,
$$f(x, y)+f(y, z) \geq f(x, z)-c_1 \Vert x-y \Vert^2-c_2 \Vert y-z \Vert^2,\;\forall x, y, z \in C.$$
\end{itemize}

The following lemmas are well known in theory of monotone equilibrium problems.

\begin{lemma}(\cite{BO})\label{2.2}\\
Let $g$ satisfy Assumption $A$. Then, for each $\alpha>0$ and $u \in \mathcal{H}$, there exists $w \in Q$ such that
$$g(w, v)+\displaystyle\frac{1}{\alpha}\langle v-w, w-u\rangle\geq 0, \;\; \forall v \in Q.$$
\end{lemma}

\begin{lemma}(\cite{CH})\label{2.3}
Under assumptions of Lemma~\ref{2.2}. Then the mapping $T_\alpha^g$ defined on $\mathcal{H}$ as follows
\begin{equation*}\label{}
T_\alpha^g(u)=\Big\{w\in Q: g(w, v)+\displaystyle\frac{1}{\alpha}\langle v-w, w-u\rangle\geq 0\;\;\forall v\in Q\Big\},
\end{equation*}
has following properties:\\
\indent (i) $T_\alpha^g$ is single-valued;\\
\indent (ii) $T_\alpha^g$ is firmly nonexpansive, i.e., for any $u, v\in \mathcal{H}$,
$$\Vert T_\alpha^g(u)-T_\alpha^g(v)\Vert^2\leq\langle T_\alpha^g(u)-T_\alpha^g(v), u-v\rangle;$$
\indent (iii) Fix$(T_\alpha^g) = \text{Sol}(Q, g)$;\\
\indent (iv) Sol$(Q, g)$ is closed and convex.
\end{lemma}

\begin{lemma}(\cite{He})\label{2.3a}
Under assumptions of Lemma~\ref{2.3}. Then for $\alpha, \beta >0$ and $u, v\in \mathcal{H}$, one has
$$\Vert T_\alpha^g(u)-T_\beta^g(v)\Vert\leq\Vert v-u\Vert+\displaystyle\frac{|\beta-\alpha|}{\beta}\Vert T_\beta^g(v)-v\Vert.$$
\end{lemma}

\section{Weak convergence theorem}
\begin{theorem}\label{wct}
Let $C$, $Q$ be two nonempty closed convex subsets in $\mathcal{H}_1$ and $\mathcal{H}_2$, respectively. Let $S: C \to  C$; $T: Q \to Q$ be nonexpansive mappings, and bifunctions $g$, $f$ satisfy assumptions $A$, assumptions $B$, respectively.  Let $A: \mathcal{H}_1 \to \mathcal{H}_2$ be a bounded linear operator with its adjoint $A^*$. Take $x^1\in C$; $\{\lambda_k\}\subset [a, b]$, for some $a, b \in \Big(0, \min\Big\{\displaystyle\frac{1}{2c_1}, \displaystyle\frac{1}{2c_2}\Big\}\Big)$; $0 < \alpha < 1$; $0 < \mu < \displaystyle \frac{1}{\Vert A\Vert^2}$; $\{\alpha_k\} \subset (0,  + \infty)$ with $\liminf \limits_{k \longrightarrow +\infty}\alpha_k>0$, and consider the sequences
$\{x^k\}$, $\{y^k\}$, $\{z^k\}$, $\{t^k\}$, and $\{u^k\}$ defined by
\begin{equation}\label{}
\begin{cases}
                                   \mbox{$y^k=\arg\min\Big\{\lambda_k f(x^k, y)+\displaystyle\frac{1}{2}\Vert y-x^k\Vert^2: y\in C\Big\}$},\\
                                   \mbox{$z^k=\arg\min\Big\{\lambda_k f(y^k, y)+\displaystyle\frac{1}{2}\Vert y-x^k\Vert^2: y\in C\Big\}$},\\
									\mbox{$t^k=(1-\alpha)z^k+\alpha Sz^k$},\\
\mbox{$u^k=T_{\alpha_k}^gAt^k$},\\
\mbox{$x^{k+1}=P_C(t^k+\mu A^*(Tu^k-At^k))$}.
                \end{cases}
\end{equation}
If $\Omega = \{x^* \in \text{Sol}(C, f) \cap \text{Fix}(S): Ax^* \in \text{Fix}(Q, g) \cap \text{Fix}(T) \} \neq \emptyset$, then the sequences  $\{x^k\}$ and $\{z^k\}$ converge weakly to an element $p \in \Omega$ and $\{u^k\}$ converges weakly to $Ap \in \text{Sol}(Q, g) \cap \text{Fix}(T)$.
\end{theorem}

Before going to prove this theorem, let us recall the following result which was proved in \cite{Anh}
\begin{lemma}(\cite{Anh})\label{a1}
Suppose that $x^* \in \text{Sol}(C, f)$, $f(x, \cdot)$ is convex and subdifferentiable on $C$ for all $x \in C$, and $f$ is pseudomonotone on $C$. Then, we have:
\begin{itemize}
\item[(i)] $\lambda_k[f(x^k, y)-f(x^k, y^k)] \geq \langle y^k-x^k, y^k-y \rangle,\; \forall y \in C.$
\item[(ii)] $\Vert z^k - x^* \Vert^2 \leq \Vert x^k-x^*\Vert^2-(1-2\lambda_k c_1)\Vert x^k - y^k \Vert^2 - (1-2\lambda_k c_2) \Vert y^k - z^k \Vert^2,\;\;\forall k.$
\end{itemize}
\end{lemma}
Now, let us prove Theorem~\ref{wct}.

\begin{proof}
Take $x^* \in \Omega$, i.e., $x^* \in \text{Sol}(C, f) \cap \text{Fix}(S)$ and $Ax^* \in \text{Sol}(Q, g) \cap \text{Fix}(T)$. By definition of $t^k$, we have
\begin{equation}\notag
\begin{aligned}
\Vert t^k-x^*\Vert&=\Vert(1-\alpha)z^k+\alpha Sz^k-x^*\Vert\\
&=\Vert (1-\alpha)(z^k-x^*)+\alpha(Sz^k-Sx^*)\Vert\\
&\leq (1-\alpha)\Vert z^k-x^*\Vert+\alpha\Vert Sz^k-Sx^*\Vert\\
&\leq (1-\alpha)\Vert z^k-x^*\Vert+\alpha\Vert z^k-x^*\Vert\\
&=\Vert z^k-x^*\Vert.
\end{aligned}
\end{equation}
Since $\{\lambda_k\}\subset [a, b]\subset\Big(0, \min\Big\{\displaystyle\frac{1}{2c_1}, \displaystyle\frac{1}{2c_2}\Big\}\Big)$ and Lemma~\ref{a1}, we get
\begin{equation}\notag
\begin{aligned}
\Vert z^k-x^*\Vert^2&\leq\Vert x^k-x^*\Vert^2-(1-2\lambda_k c_1)\Vert x^k-y^k\Vert^2-(1-2\lambda_k c_2)\Vert y^k-z^k\Vert^2\\
&\leq\Vert x^k-x^*\Vert^2.
\end{aligned}
\end{equation}
Thus
\begin{equation}\label{p1}
\Vert t^k-x^*\Vert\leq\Vert z^k-x^*\Vert\leq\Vert x^k-x^*\Vert.
\end{equation}
it follows from Lemma~\ref{2.3} that
\begin{equation}\notag
\begin{aligned}
\Vert T_{\alpha_k}^gAt^k-Ax^*\Vert^2&=\Vert T_{\alpha_k}^gAt^k-T_{\alpha_k}^gAx^*\Vert^2\\
&\leq\langle T_{\alpha_k}^gAt^k-T_{\alpha_k}^gAx^*, At^k-Ax^*\rangle\\
&=\langle T_{\alpha_k}^gAt^k-Ax^*, At^k-Ax^*\rangle\\
&=\displaystyle\frac{1}{2}\Big[\Vert T_{\alpha_k}^gAt^k-Ax^*\Vert^2+\Vert At^k-Ax^*\Vert^2
-\Vert T_{\alpha_k}^gAt^k-At^k\Vert^2\Big].
\end{aligned}
\end{equation}
Hence
$$\Vert T_{\alpha_k}^gAt^k-Ax^*\Vert^2\leq\Vert At^k-Ax^*\Vert^2-\Vert T_{\alpha_k}^gAt^k-At^k\Vert^2.$$
Combining this fact with the nonexpansiveness of the mapping  $T$, one has
\begin{equation}\label{p2}
\begin{aligned}
\Vert Tu^k-Ax^*\Vert^2&=\Vert TT_{\alpha_k}^gAt^k-TAx^*\Vert^2\\
&\leq\Vert T_{\alpha_k}^gAt^k-Ax^*\Vert^2\\
&\leq\Vert At^k-Ax^*\Vert^2-\Vert T_{\alpha_k}^gAt^k-At^k\Vert^2.
\end{aligned}
\end{equation}
Using \eqref{p2}, we obtain
\begin{equation}\label{p3}
\begin{aligned}
\langle A(t^k-x^*), Tu^k-At^k\rangle&=\langle A(t^k-x^*)+Tu^k-At^k-(Tu^k-At^k), Tu^k-At^k\rangle\\
&=\langle Tu^k-Ax^*, Tu^k-At^k\rangle-\Vert Tu^k-At^k\Vert^2\\
&=\displaystyle\frac{1}{2}\Big[\Vert Tu^k-Ax^*\Vert^2+\Vert Tu^k-At^k\Vert^2-\Vert At^k-Ax^*\Vert^2\Big]-\Vert Tu^k-At^k\Vert^2\\
&=\displaystyle\frac{1}{2}\Big[(\Vert Tu^k-Ax^*\Vert^2-\Vert At^k-Ax^*\Vert^2)-\Vert Tu^k-At^k\Vert^2\Big]\\
&\leq -\displaystyle\frac{1}{2}\Vert T_{\alpha_k}^gAt^k-At^k\Vert^2-\displaystyle\frac{1}{2}\Vert Tu^k-At^k\Vert^2.
\end{aligned}
\end{equation}
It implies from \eqref{p1} and \eqref{p3} that
\begin{equation}\label{a}
\begin{aligned}
\Vert x^{k+1}-x^*\Vert^2&=\Vert P_C(t^k+\mu A^*(Tu^k-At^k))-P_C(x^*)\Vert^2\\
&\leq\Vert (t^k-x^*)+\mu A^*(Tu^k-At^k)\Vert^2\\
&=\Vert t^k-x^*\Vert^2+\Vert \mu A^*(Tu^k-At^k)\Vert^2+2\mu\langle t^k-x^*, A^*(Tu^k-At^k)\rangle\\
&\leq\Vert t^k-x^*\Vert^2+\mu^2\Vert A^*\Vert^2\Vert Tu^k-At^k\Vert^2+2\mu\langle A(t^k-x^*), Tu^k-At^k\rangle\\
&\leq\Vert t^k-x^*\Vert^2+\mu^2\Vert A^*\Vert^2\Vert Tu^k-At^k\Vert^2-\mu\Vert T_{\alpha_k}^gAt^k-At^k\Vert^2-\mu\Vert Tu^k-At^k\Vert^2\\
&=\Vert t^k-x^*\Vert^2-\mu(1-\mu\Vert A^*\Vert^2)\Vert Tu^k-At^k\Vert^2-\mu\Vert T_{\alpha_k}^gAt^k-At^k\Vert^2\\
&=\Vert t^k-x^*\Vert^2-\mu(1-\mu\Vert A^*\Vert^2)\Vert Tu^k-At^k\Vert^2-\mu\Vert u^k-At^k\Vert^2\\
&\leq\Vert x^k-x^*\Vert^2-\mu(1-\mu\Vert A^*\Vert^2)\Vert Tu^k-At^k\Vert^2-\mu\Vert u^k-At^k\Vert^2\\
&=\Vert x^k-x^*\Vert^2-\mu(1-\mu\Vert A\Vert^2)\Vert Tu^k-At^k\Vert^2-\mu\Vert u^k-At^k\Vert^2.
\end{aligned}
\end{equation}
In view of \eqref{p1}, \eqref{a} and $0<\mu<\displaystyle\frac{1}{\Vert A\Vert^2}$, yields
\begin{equation}\label{p4}
\Vert x^{k+1}-x^*\Vert\leq\Vert t^k-x^*\Vert\leq\Vert z^k-x^*\Vert\leq\Vert x^k-x^*\Vert,
\end{equation}
and
\begin{equation}\label{p5}
\mu(1-\mu\Vert A\Vert^2)\Vert Tu^k-At^k\Vert^2+\mu\Vert u^k-At^k\Vert^2\leq\Vert x^k-x^*\Vert^2-\Vert x^{k+1}-x^*\Vert^2.
\end{equation}
Since \eqref{p4}, we conclude that $\lim\limits_{k \to +\infty} \Vert x^k-x^* \Vert$ exists. So, we receive  from \eqref{p5} that,
\begin{equation}\label{p6}
\begin{aligned}
\lim\limits_{k \to +\infty}\Vert x^k-x^*\Vert&=\lim\limits_{k \to +\infty}\Vert t^k-x^*\Vert=\lim\limits_{k\to +\infty}\Vert z^k-x^*\Vert, \text{ and } \\
\lim\limits_{k \to +\infty}\Vert Tu^k-At^k\Vert&=\lim\limits_{k \to  +\infty}\Vert u^k-At^k\Vert=0.
\end{aligned}
\end{equation}
From \eqref{p6} and the inequality
$$\Vert Tu^k-u^k\Vert\leq \Vert Tu^k-At^k\Vert+\Vert u^k-At^k\Vert,$$
we get
\begin{equation}\label{p7}
\lim\limits_{k\longrightarrow +\infty}\Vert Tu^k-u^k\Vert=0.
\end{equation}
Besides that,
$$\Vert z^k-x^*\Vert^2 \leq\Vert x^k-x^*\Vert^2-(1-2\lambda_k c_1)\Vert x^k-y^k\Vert^2-(1-2\lambda_k c_2)\Vert y^k-z^k\Vert^2,$$
so
\begin{equation}\label{p8}
(1-2\lambda_k c_1)\Vert x^k-y^k\Vert^2+(1-2\lambda_k c_2)\Vert y^k-z^k\Vert^2\leq\Vert x^k-x^*\Vert^2-\Vert z^k-x^*\Vert^2.
\end{equation}
Since $\{\lambda_k\}\subset [a, b]\subset\Big(0, \min\Big\{\displaystyle\frac{1}{2c_1}, \displaystyle\frac{1}{2c_2}\Big\}\Big)$ and \eqref{p8}, one has
\begin{equation}\label{p9}
\begin{aligned}
(1-2bc_1)\Vert x^k-y^k\Vert^2\leq&\Vert x^k-x^*\Vert^2-\Vert z^k-x^*\Vert^2,\\
(1-2bc_2)\Vert y^k-z^k\Vert^2\leq&\Vert x^k-x^*\Vert^2-\Vert z^k-x^*\Vert^2.
\end{aligned}
\end{equation}
From \eqref{p6}, we get $\lim\limits_{k \to +\infty}(\Vert x^k-x^*\Vert^2-\Vert z^k-x^*\Vert^2)=0$. \\
Combining this fact with \eqref{p9}, yields
\begin{equation}\label{p10}
\lim\limits_{k\longrightarrow +\infty}\Vert x^k-y^k\Vert=0,\;\;\lim\limits_{k\longrightarrow +\infty}\Vert y^k-z^k\Vert=0.
\end{equation}
It is clear that
$$\Vert z^k-x^k \Vert \leq \Vert x^k-y^k \Vert + \Vert y^k-z^k \Vert, $$
so, we get from \eqref{p10} that
\begin{equation}\label{p11}
\lim\limits_{k \to +\infty}\Vert z^k-x^k\Vert=0.
\end{equation}
Using $t^k=(1-\alpha)z^k+\alpha Sz^k$, Lemma~\ref{2.5} and the nonexpansiveness of $S$, we have
\begin{equation}\label{b}
\begin{aligned}
\Vert t^k-x^*\Vert^2&=\Vert(1-\alpha)z^k+\alpha Sz^k-x^*\Vert^2\\
&=\Vert (1-\alpha)(z^k-x^*)+\alpha(Sz^k-x^*)\Vert^2\\
&=(1-\alpha)\Vert z^k-x^*\Vert^2+\alpha\Vert Sz^k-x^*\Vert^2-\alpha(1-\alpha)\Vert Sz^k-z^k\Vert^2\\
&=(1-\alpha)\Vert z^k-x^*\Vert^2+\alpha\Vert Sz^k-Sx^*\Vert^2-\alpha(1-\alpha)\Vert Sz^k-z^k\Vert^2\\
&\leq (1-\alpha)\Vert z^k-x^*\Vert^2+\alpha\Vert z^k-x^*\Vert^2-\alpha(1-\alpha)\Vert Sz^k-z^k\Vert^2\\
&=\Vert z^k-x^*\Vert^2-\alpha(1-\alpha)\Vert Sz^k-z^k\Vert^2.
\end{aligned}
\end{equation}
Therefore,
$$\alpha(1-\alpha)\Vert Sz^k-z^k\Vert^2\leq\Vert z^k-x^*\Vert^2-\Vert t^k-x^*\Vert^2.$$
Combining the last inequality with \eqref{p6}, we conclude that
\begin{equation}\label{p12}
\lim\limits_{k\longrightarrow +\infty}\Vert Sz^k-z^k\Vert=0.
\end{equation}
In addition,
\begin{equation}\notag
\begin{aligned}
\Vert t^k-x^k\Vert&\leq \Vert t^k-z^k\Vert+\Vert z^k-x^k\Vert\\
&=\alpha \Vert Sz^k-z^k \Vert + \Vert z^k-x^k \Vert.
\end{aligned}
\end{equation}
Therefore, we receive from \eqref{p11} and \eqref{p12} that
\begin{equation}\label{p13}
\lim\limits_{k\longrightarrow +\infty}\Vert t^k-x^k\Vert=0.
\end{equation}

Because $\lim\limits_{k \to +\infty}\Vert x^k-x^*\Vert$ exists, $\{x^k\}$ is bounded. Consequently, there exists a subsequence $\{x^{k_j}\}$ of
$\{x^k\}$ such that $x^{k_j}$ converges weakly to some $p \in C$ as $j \to +\infty$. Then, it follows from \eqref{p13} that $t^{k_j}\rightharpoonup p$, $At^{k_j}\rightharpoonup Ap$.  \\
Since $\lim\limits_{k\longrightarrow +\infty}\Vert u^k-At^k\Vert=0$, we deduce that $u^{k_j} \rightharpoonup Ap$. Remember that $\{u^k\} \subset Q$, so $Ap \in Q$.\\
In addition, $\lim\limits_{k \to +\infty}\Vert z^k-x^k\Vert=0$ and $x^{k_j}\rightharpoonup p$. Hence, $z^{k_j}\rightharpoonup p$.\\

If $Sp \neq p$, then, by Opial's condition and \eqref{p12}, we have
\begin{equation}\notag
\begin{aligned}
\liminf\limits_{j\longrightarrow +\infty}\Vert z^{k_j}-p\Vert&<\liminf\limits_{j\longrightarrow +\infty}\Vert z^{k_j}-Sp\Vert\\
&=\liminf\limits_{j\longrightarrow +\infty}\Vert z^{k_j}-Sz^{k_j}+Sz^{k_j}-Sp\Vert\\
&\leq\liminf\limits_{j\longrightarrow +\infty}(\Vert z^{k_j}-Sz^{k_j}\Vert+\Vert Sz^{k_j}-Sp\Vert)\\
&=\liminf\limits_{j\longrightarrow +\infty}\Vert Sz^{k_j}-Sp\Vert\\
&\leq\liminf\limits_{j\longrightarrow +\infty}\Vert z^{k_j}-p\Vert,
\end{aligned}
\end{equation}
this is a contradiction. So $Sp=p$, i.e., $p \in \text{Fix}(S)$.\\
From Lemma~\ref{a1}
$$\lambda_{k_j}[f(x^{k_j}, y)-f(x^{k_j}, y^{k_j})]\geq\langle y^{k_j}-x^{k_j}, y^{k_j}-y\rangle,\;\forall y\in C.$$
Letting $j \to +\infty$, we get $f(p, y) \geq 0$ for all $y \in C$. It means that $p \in \text{Sol}(C, f)$. \\
Therefore,
\begin{equation}\label{p14}
p \in \text{Sol}(C, f) \cap \text{Fix}(S).
\end{equation}
Next, we need showing that $Ap \in \text{Sol}(Q, g) \cap \text{Fix}(T)$. \\
Indeed, if $TAp \neq Ap$, then by Opial's condition and \eqref{p7}, we have
\begin{equation}\notag
\begin{aligned}
\liminf\limits_{j \to +\infty}\Vert u^{k_j}-Ap\Vert&<\liminf\limits_{j \to +\infty}\Vert u^{k_j}-TAp\Vert\\
&=\liminf\limits_{j \to +\infty}\Vert u^{k_j}-Tu^{k_j}+Tu^{k_j}-TAp\Vert\\
&\leq\liminf\limits_{j \to +\infty}(\Vert u^{k_j}-Tu^{k_j}\Vert+\Vert Tu^{k_j}-TAp\Vert)\\
&=\liminf\limits_{j \to +\infty}\Vert Tu^{k_j}-TAp\Vert\\
&\leq\liminf\limits_{j \to +\infty}\Vert u^{k_j}-Ap\Vert,
\end{aligned}
\end{equation}
this is a contradiction. Thus, $Ap \in \text{Fix}(T)$.\\
On the other hand, $\text{Sol}(Q, g)= \text{Fix}(T_\alpha^g)$. So, if $T_r^gAp \neq Ap$, then combining with \eqref{p6}, Opial's condition, and Lemma~\ref{2.3a}, we have
\begin{equation}\notag
\begin{aligned}
\liminf\limits_{j \to +\infty}\Vert At^{k_j}-Ap\Vert&<\liminf\limits_{j \to +\infty}\Vert At^{k_j}-T_\alpha^gAp\Vert\\
&=\liminf\limits_{j \to +\infty}\Vert At^{k_j}-u^{k_j}+u^{k_j}-T_\alpha^gAp\Vert\\
&\leq\liminf\limits_{j \to +\infty}(\Vert At^{k_j}-u^{k_j}\Vert+\Vert T_\alpha^gAp-u^{k_j}\Vert)\\
&=\liminf\limits_{j \to +\infty}\Vert T_\alpha^gAp-u^{k_j}\Vert\\
&=\liminf\limits_{j \to +\infty}\Vert T_\alpha^gAp-T_{\alpha_{k_j}}^gAt^{k_j}\Vert\\
&\leq\liminf\limits_{j \to +\infty}\Big\{\Vert At^{k_j}-Ap\Vert+\displaystyle\frac{|\alpha_{k_j}-\alpha|}{\alpha_{k_j}}\Vert T_{\alpha_{k_j}}^gAt^{k_j}-At^{k_j}\Vert\Big\}\\
&=\liminf\limits_{j \to +\infty}\Big\{\Vert At^{k_j}-Ap\Vert+\displaystyle\frac{|\alpha_{k_j}-\alpha|}{\alpha_{k_j}}\Vert u^{k_j}-At^{k_j}\Vert\Big\}\\
&=\liminf\limits_{j \to +\infty}\Vert At^{k_j}-Ap\Vert.
\end{aligned}
\end{equation}
this is a contradiction. Thus $Ap \in \text{Fix}(T_\alpha^g) = \text{Sol}(Q, g)$.\\
 Therefore,
\begin{equation}\label{p15}
Ap \in \text{Sol}(Q, g) \cap \text{Fix}(T).
\end{equation}
From \eqref{p14} and \eqref{p15}, we obtain $p \in \Omega$.\\
Finally, we prove $\{x^k\}$ converges weakly to $p$. Otherwise, there exists a subsequence $\{x^{m_i}\}$ of $\{x^k\}$ such that $x^{m_i} \rightharpoonup q \in \Omega$ with $q \neq p$, then by Opial's condition, yields
\begin{equation}\notag
\begin{aligned}
\liminf\limits_{i \to +\infty}\Vert x^{m_i} - q \Vert & <\liminf\limits_{i \to  +\infty}\Vert x^{m_i} - p\Vert\\
&=\liminf\limits_{j \to +\infty}\Vert x^{k_j} - p \Vert\\
&<\liminf\limits_{j \to +\infty}\Vert x^{k_j}-q \Vert\\
&=\liminf\limits_{i \to +\infty}\Vert x^{m_i}-q\Vert.
\end{aligned}
\end{equation}
This is a contradiction. Hence $\{x^k\}$ converges weakly to $p$. \\Together with \eqref{p11} and \eqref{p13}, we also get $z^k\rightharpoonup p$ and $t^k \rightharpoonup p$, so $At^k \rightharpoonup Ap$. Combining with \eqref{p6}, it is immediate that $u^k \rightharpoonup Ap \in \text{Sol}(Q, g) \cap \text{Fix}(T)$.
%\hfill$\Box$
\end{proof}

When $S = I_{\mathcal{H}_1}$ and $T = I_{\mathcal{H}_2}$, the problem (SEPNM) reduces to the split equilibrium problem (SEP). In this case, Theorem~\ref{wct} becomes

\begin{corollary}\label{corollary1}
Suppose that $g$, $f$ are bifunctions satisfying assumptions $A$ and assumptions $B$, respectively. Let $A: \mathcal{H}_1 \to \mathcal{H}_2$ be a bounded linear operator with its adjoint $A^*$. Take $x^1\in C$; $\{\lambda_k\} \subset [a, b]$, for some $a, b \in \Big(0, \min\Big\{\displaystyle\frac{1}{2c_1}, \displaystyle\frac{1}{2c_2}\Big\}\Big)$; $0 < \alpha < 1$; $0 < \mu < \displaystyle \frac{1}{\Vert A\Vert^2}$; $\{\alpha_k\} \subset (0,  + \infty)$ with $\liminf \limits_{k \to +\infty}\alpha_k>0$, and consider the sequences
$\{x^k\}$, $\{y^k\}$, $\{z^k\}$, and $\{u^k\}$ defined by

\begin{equation}\label{}
\begin{cases}
                                   \mbox{$y^k=\arg\min\Big\{\lambda_k f(x^k, y)+\displaystyle\frac{1}{2}\Vert y-x^k\Vert^2: y\in C\Big\}$},\\
                                   \mbox{$z^k=\arg\min\Big\{\lambda_k f(y^k, y)+\displaystyle\frac{1}{2}\Vert y-x^k\Vert^2: y\in C\Big\}$},\\
									
\mbox{$u^k=T_{\alpha_k}^gAz^k$},\\
\mbox{$x^{k+1}=P_C(z^k+\mu A^*(Tu^k-Az^k))$}.
                \end{cases}
\end{equation}
If  $\Omega = \{x^* \in \text{Sol}(C, f): Ax^* \in \text{Sol}(Q, g) \} \neq \emptyset$, then the sequences  $\{x^k\}$ and $\{z^k\}$ converge weakly to an element $p \in \Omega$ and $\{u^k\}$ converges weakly to $Ap \in \text{Sol}(Q, g)$.
\end{corollary}

\section{Strong convergence theorem}
\begin{theorem}\label{sct}

Let $x^1\in C_1= C$, consider sequences $\{x^k\}$, $\{y^k\}$, $\{z^k\}$, $\{t^k\}$ and $\{u^k\}$  generated by the following process
\begin{equation}\label{s1}
\begin{cases}
                                   \mbox{$y^k=\arg\min\Big\{\lambda_k f(x^k, y)+\displaystyle\frac{1}{2}\Vert y-x^k\Vert^2: y\in C\Big\}$},\\
                                   \mbox{$z^k=\arg\min\Big\{\lambda_k f(y^k, y)+\displaystyle\frac{1}{2}\Vert y-x^k\Vert^2: y\in C\Big\}$},\\
									\mbox{$t^k=(1-\alpha)z^k+\alpha Sz^k$},\\
\mbox{$u^k=T_{\alpha_k}^gAt^k$},\\
\mbox{$s^k=P_C(t^k+\mu A^*(Tu^k-At^k))$},\\
\mbox{$C_{k+1}=\{r\in C_k: \Vert s^k-r\Vert\leq\Vert t^k-r\Vert\leq\Vert x^k-r\Vert\}$},\\
\mbox{$x^{k+1}=P_{C_{k+1}}(x^1),\;\;\;k\in\Bbb{N^*}$}
                \end{cases}
\end{equation}
where $0 < \alpha <1$, $0 < \mu < \displaystyle\frac{1}{\Vert A\Vert^2}$, $\{\alpha_k\} \subset (0; +\infty)$ with $\liminf\limits_{k \to +\infty} \alpha_k>0$. Then under assumptions of Theorem~\ref{wct} and $\Omega=\{x^*\in \text{Sol}(C, f) \cap \text{Fix}(S): Ax^* \in \text{Sol}(Q, g) \cap \text{Fix}(T)\} \neq \emptyset$, the sequences $\{x^k\}$, $\{z^k\}$ converge strongly to an element $p \in \Omega$ and $\{u^k\}$ converges strongly to $Ap \in \text{Sol}(Q, g) \cap \text{Fix}(T)$.
\end{theorem}
\begin{proof}
Firstly, We claim that $C_k$ is a nonempty closed convex set for all $k \in \Bbb{N^*}$.
In fact, let $x^* \in \Omega$, it follows from \eqref{p1}, \eqref{a}, \eqref{b} that
\begin{equation}\label{s2}
\begin{aligned}
\Vert s^k-x^*\Vert^2&\leq\Vert t^k-x^*\Vert^2-\mu(1-\mu\Vert A\Vert^2)\Vert Tu^k-At^k\Vert^2-\mu\Vert u^k-At^k\Vert^2\\
&\leq\Vert z^k-x^*\Vert^2-\alpha(1-\alpha)\Vert Sz^k-z^k\Vert^2-\mu(1-\mu\Vert A\Vert^2)\Vert Tu^k-At^k\Vert^2\\
&\;\;\;\;\;\;\;\;\;\;\;\;\;\;\;\;\;\;\;\;\;\;-\mu\Vert u^k-At^k\Vert^2\\
&\leq\Vert x^k-x^*\Vert^2-\alpha(1-\alpha)\Vert Sz^k-z^k\Vert^2-\mu(1-\mu\Vert A\Vert^2)\Vert Tu^k-At^k\Vert^2\\
&\;\;\;\;\;\;\;\;\;\;\;\;\;\;\;\;\;\;\;\;\;\;-\mu\Vert u^k-At^k\Vert^2\\
\end{aligned}
\end{equation}
By the algorithm $0 < \mu< \displaystyle\frac{1}{\Vert A\Vert^2}$, and \eqref{s2}, we have
\begin{equation}\label{s3}
\Vert s^k-x^*\Vert\leq\Vert t^k-x^*\Vert\leq\Vert z^k-x^*\Vert\leq\Vert x^k-x^*\Vert,\;\;\forall k.
\end{equation}
Since $x^* \in C_1$ and \eqref{s3}, we get by induction that $x^* \in C_k$ for all $k \in \Bbb{N^*}$, i.e., $\Omega \subset C_k$, so $C_k \neq \emptyset$ for all $k$.\\
Define
$$D_k=\{r\in \mathcal{H}_1: \Vert s^k-r\Vert\leq\Vert t^k-r\Vert\leq\Vert x^k-r\Vert\}\; k\in \Bbb{N^*},$$
then $C_{k+1}=C_k \cap D_k$. Because $C_1$ and $D_k$ are closed for all $k$,  $C_k$ is closed for all $k$.\\
Next, we verify that $C_k$ is convex for all $k$. Indeed, let $r^1, r^2 \in C_{k+1}$ and $\lambda \in [0, 1]$, using Lemma~\ref{2.5}, we have
\begin{equation}\notag
\begin{aligned}
\Vert s^k-(\lambda r^1+(1-\lambda)r^2)\Vert^2&=\Vert\lambda(s^k-r^1)+(1-\lambda)(s^k-r^2)\Vert^2\\
&=\lambda\Vert s^k-r^1\Vert^2+(1-\lambda)\Vert s^k-r^2\Vert^2-\lambda(1-\lambda)\Vert r^1-r^2\Vert^2\\
&\leq\lambda\Vert t^k-r^1\Vert^2+(1-\lambda)\Vert t^k-r^2\Vert^2-\lambda(1-\lambda)\Vert r^1-r^2\Vert^2\\
&=\Vert t^k-(\lambda r^1+(1-\lambda)r^2)\Vert^2.
\end{aligned}
\end{equation}
So, $$\Vert s^k-(\lambda r^1+(1-\lambda)r^2)\Vert\leq\Vert t^k-(\lambda r^1+(1-\lambda)r^2)\Vert.$$
Similarly,  $$\Vert t^k-(\lambda r^1+(1-\lambda)r^2)\Vert\leq\Vert x^k-(\lambda r^1+(1-\lambda)r^2)\Vert.$$
Thus
$$\Vert s^k-(\lambda r^1+(1-\lambda)r^2)\Vert\leq\Vert t^k-(\lambda r^1+(1-\lambda)r^2)\Vert\leq\Vert x^k-(\lambda r^1+(1-\lambda)r^2)\Vert.$$
 Therefore,
$$\lambda r^1+(1-\lambda)r^2\in C_{k+1}.$$
%and $C_{k+1}$ is convex set for all $k$.\\
Notice that $x^{k+1}\in C_{k+1}\subset C_k$ and $x^k=P_{C_k}(x^1)$, so
$$\Vert x^k-x^1\Vert\leq\Vert x^{k+1}-x^1\Vert,  \text{ for all } k.$$
In addition,  $x^{k+1}=P_{C_{k+1}}(x^1)$ and $x^* \in C_{k+1}$, it implies that
$$\Vert x^{k+1}-x^1 \Vert \leq \Vert x^*-x^1 \Vert.$$
Thus
$$\Vert x^k-x^1 \Vert \leq \Vert x^{k+1}-x^1 \Vert \leq \Vert x^*-x^1 \Vert, \;\forall k.$$
Therefore, $\lim\limits_{k \to +\infty}\Vert x^k-x^1\Vert$ exists, consequently $\{x^k\}$ is bounded. \\
Hence, $\{s^k\}$ and $\{t^k\}$ are also bounded.\\
For all $m>n$, we have $x^m \in C_m \subset C_n$, $x^n=P_{C_n}(x^1)$. Combining this fact with Lemma~\ref{2.1}, we get
$$\Vert x^m-x^n\Vert^2\leq\Vert x^m-x^1\Vert^2-\Vert x^n-x^1\Vert^2.$$
Since $\lim\limits_{k \to +\infty}\Vert x^k-x^1\Vert$ exists, it implies that $\{x^k\}$ is a Cauchy sequence, i.e.,
\begin{equation}\label{s3.t}
\lim_{k \to \infty}x^k = p.
\end{equation}
We need showing that $p \in \Omega$. From the definitions of $C_{k+1}$ and $x^{k+1}$, we have
$$\Vert s^k-x^{k+1}\Vert\leq\Vert t^k-x^{k+1}\Vert\leq\Vert x^k-x^{k+1}\Vert.$$
Thus
\begin{equation}\label{s4}
\begin{aligned}
\Vert s^k-x^k\Vert&\leq\Vert s^k-x^{k+1}\Vert+\Vert x^{k+1}-x^k\Vert\\
&\leq\Vert x^k-x^{k+1}\Vert+\Vert x^k-x^{k+1}\Vert\\
&=2\Vert x^k-x^{k+1}\Vert,
\end{aligned}
\end{equation}
and
\begin{equation}\label{s5}
\begin{aligned}
\Vert t^k-x^k\Vert&\leq\Vert t^k-x^{k+1}\Vert+\Vert x^{k+1}-x^k\Vert\\
&\leq\Vert x^k-x^{k+1}\Vert+\Vert x^k-x^{k+1}\Vert\\
&=2\Vert x^k-x^{k+1}\Vert.
\end{aligned}
\end{equation}
Combining with \eqref{s3.t}, we deduce from \eqref{s4} and \eqref{s5} that

\begin{equation}\label{s5s6}
\lim_{k \to \infty} \|s^k - x^k \| = \lim_{k \to \infty} \|t^k - x^k \| = 0.
\end{equation}
In view of \eqref{s2} and \eqref{s5s6}, one has
\begin{equation}\label{s6}
\begin{aligned}
\alpha(1-\alpha)\Vert& Sz^k-z^k\Vert^2+\mu(1-\mu\Vert A\Vert^2)\Vert Tu^k-At^k\Vert^2+\mu\Vert u^k-At^k\Vert^2\\
&\leq\Vert x^k-x^*\Vert^2-\Vert s^k-x^*\Vert^2\\
&=(\Vert x^k-x^*\Vert+\Vert s^k-x^*\Vert)(\Vert x^k-x^*\Vert-\Vert s^k-x^*\Vert)\\
&\leq\Vert x^k-s^k\Vert(\Vert x^k-x^*\Vert+\Vert s^k-x^*\Vert) \to 0 \text{ as } k \to \infty.
\end{aligned}
\end{equation}
Since $0 < \alpha < 1$ and $\mu \in (0, \frac{1}{\|A\|})$, we get from \eqref{s6} that
\begin{equation}\label{s7}
\begin{aligned}
\lim\limits_{k \to +\infty}\Vert Tu^k-At^k\Vert&=\lim\limits_{k \to +\infty}\Vert u^k-At^k\Vert=0, \text{ and }\\
\lim\limits_{k \to +\infty}\Vert Sz^k-z^k\Vert&=0.
\end{aligned}
\end{equation}
\eqref{s7} and the inequality
$$\Vert Tu^k-u^k\Vert\leq \Vert Tu^k-At^k\Vert+\Vert u^k-At^k\Vert$$
also imply that
\begin{equation}\label{s8}
\lim\limits_{k \to +\infty}\Vert Tu^k-u^k\Vert=0.
\end{equation}
From \eqref{p10}, \eqref{p11}, \eqref{p13} and $\lim\limits_{k \to +\infty}x^k=p$, we have
\begin{equation}\label{s9}
\lim\limits_{k \to +\infty}y^k=p,\;\;\lim\limits_{k \to +\infty}z^k=p,\;\;\lim\limits_{k \to +\infty}t^k=p.
\end{equation}
It follows from \eqref{s7} and \eqref{s9} that
\begin{equation}\label{}
\begin{aligned}
\Vert Sp-p\Vert&\leq\Vert Sp-Sz^k\Vert+\Vert Sz^k-z^k\Vert+\Vert z^k-p\Vert\\
&\leq\Vert p-z^k\Vert+\Vert Sz^k-z^k\Vert+\Vert z^k-p\Vert\\
&=2\Vert z^k-p\Vert+\Vert Sz^k-z^k\Vert \to 0, \text{ as } k \to \infty.
\end{aligned}
\end{equation}
So, $Sp = p$ i.e., $p \in \text{Fix}(S)$.\\
In the other side, we receive from Lemma~\ref{a1} that
$$\lambda_k[f(x^k, y) - f(x^k, y^k)]\geq\langle y^k - x^k, y^k-y\rangle, \;\forall y\in C.$$
Letting $k \to +\infty$, by the joint weak continuity of $f$, $\lim\limits_{k \to +\infty}x^k=p$ and \eqref{s9}, we get in the limit that
$$f(p, y)\geq 0 \text{ for all } y \in C.$$
It is immediate that $p \in \text{Sol}(C, f)$. Consequently,
\begin{equation}\label{s10}
p \in \text{Sol}(C, f) \cap \text{Fix}(S).
\end{equation}
Since \eqref{s9}, it implies that $\lim\limits_{k \to +\infty}At^k=Ap$. Combinating this with \eqref{s7}, one has
\begin{equation}\label{s11}
\lim\limits_{k\longrightarrow +\infty}u^k=Ap.
\end{equation}
From \eqref{s8}, \eqref{s11}, it yields
\begin{equation}\label{}
\begin{aligned}
\Vert TAp-Ap\Vert&\leq\Vert TAp-Tu^k\Vert+\Vert Tu^k-u^k\Vert+\Vert u^k-Ap\Vert\\
&\leq\Vert Ap-u^k\Vert+\Vert Tu^k-u^k\Vert+\Vert u^k-Ap\Vert\\
&=2\Vert u^k-Ap\Vert+\Vert Tu^k-u^k\Vert \to 0 \text{ as } k \to \infty.
\end{aligned}
\end{equation}
Hence, $TAp=Ap$, i.e., $Ap \in \text{Fix}(T)$.\\
In addition, we receive from Lemma~\ref{2.3a}, $\lim\limits_{k \to +\infty}At^k=Ap$, and \eqref{s7} that
\begin{equation*}\label{}
\begin{aligned}
\Vert T_\alpha^gAp-Ap\Vert&\leq\Vert T_\alpha^gAp-T_{\alpha_k}^gAt^k\Vert+\Vert T_{\alpha_k}^gAt^k-At^k\Vert+\Vert At^k-Ap\Vert\\
&=\Vert T_\alpha^gAp-T_{\alpha_k}^gAt^k\Vert+\Vert u^k-At^k\Vert+\Vert At^k-Ap\Vert\\
&\leq\Vert At^{k}-Ap\Vert+\displaystyle\frac{|\alpha_k-\alpha|}{\alpha_k}\Vert T_{\alpha_k}^gAt^k-At^k\Vert+\Vert u^k-At^k\Vert+\Vert At^k-Ap\Vert\\
&=2\Vert At^{k}-Ap\Vert+\displaystyle\frac{|\alpha_k-\alpha|}{\alpha_k}\Vert u^k-At^k\Vert+\Vert u^k-At^k\Vert \to 0 \text{ as } k \to \infty,
\end{aligned}
\end{equation*}
which implies that $T_\alpha^gAp=Ap$, namely $Ap \in \text{Fix}(T_\alpha^g) = \text{Sol}(Q, g)$. \\
So,
\begin{equation}\label{s12}
Ap \in \text{Sol}(Q, g) \cap \text{Fix}(T).
\end{equation}
From \eqref{s10} and \eqref{s12}, we get $p \in \Omega$. The proof is completed.
\end{proof}
The following corollary is immediate from Theorem~\ref{sct} when $S=I_{\mathcal{H}_1}$ and $T=I_{\mathcal{H}_2}$.
 \begin{corollary}\label{corollary2}
Let $g: Q \times Q \to \Bbb{R}$ be a bifunction satisfying assumptions $A$ and $f: C \times C \to \Bbb{R}$ be a bifunction satisfying assumptions $B$. Let $A: \mathcal{H}_1 \to \mathcal{H}_2$ be a bounded linear operator with its adjoint $A^*$. Choose $x^1 \in C$ and $C_1 = C$. Consider the sequences $\{x^k\}$, $\{y^k\}$, $\{z^k\}$ and $\{u^k\}$ generated by the following iteration
\begin{equation}\label{s1}
\begin{cases}
                                   \mbox{$y^k=\arg\min\Big\{\lambda_k f(x^k, y)+\displaystyle\frac{1}{2}\Vert y-x^k\Vert^2: y\in C\Big\}$},\\
                                   \mbox{$z^k=\arg\min\Big\{\lambda_k f(y^k, y)+\displaystyle\frac{1}{2}\Vert y-x^k\Vert^2: y\in C\Big\}$},\\
\mbox{$u^k=T_{\alpha_k}^gAz^k$},\\
\mbox{$s^k=P_C(z^k+\mu A^*(u^k-Az^k))$},\\
\mbox{$C_{k+1}=\{r\in C_k: \Vert s^k-r\Vert\leq\Vert z^k-r\Vert\leq\Vert x^k-r\Vert\}$},\\
\mbox{$x^{k+1}=P_{C_{k+1}}(x^1).$}
                \end{cases}
\end{equation}
where $0< \alpha <1$, $0 < \mu < \displaystyle \frac{1}{\Vert A\Vert^2}$, $\{\alpha_k\}\subset (0; +\infty)$ with $\liminf\limits_{k \to +\infty} \alpha_k>0$.  Suppose that
$\Omega=\{x^*\in \text{Sol}(C, f): Ax^* \in \text{Sol}(Q, g)\}\neq \emptyset$, then the sequences $\{x^k\}$ and $\{z^k\}$ converge strongly to an element $p \in \Omega$ and $\{u^k\}$ converges strongly to $Ap \in \text{Sol}(Q, g)$.
\end{corollary}

 {\bf Conclusion}.
We have proposed two algorithms for solving a split equilibrium and nonexpansive mapping SEPNM($C, Q, A, f, g, S, T)$ in Hilbert spaces. In which, the bifunction $f$ is pseudomonotone on $C$, the bifunction $g$ is monotone on $Q$ and $S, T$ are nonexpansive mappings. Then, we have proved that the iteration sequences generated by the algorithms converge weakly and strongly to a solution of this problem, respectively.

\begin{acknowledgements}
The first author is supported in part by NAFOSTED, under the project 101.01-2014-24.
\end{acknowledgements}

\bibliographystyle{spbasic}      % basic style, author-year citations
\bibliographystyle{spmpsci}      % mathematics and physical sciences
\bibliographystyle{spphys}       % APS-like style for physics

%\bibliographystyle{amsplain}

\end{document}